\documentclass[11pt]{article}
\usepackage[arrow, matrix]{xy}
\usepackage{amsmath}
\usepackage{cancel}
\usepackage{amssymb}
\usepackage{amsthm}
\usepackage[mathscr]{eucal}
\input{epsf}
\theoremstyle{definition}
\newtheorem{definition}{Definition}
\newtheorem{theorem}[definition]{Theorem}

\theoremstyle{remark}
\newtheorem{remark}[definition]{Remark}

\newcounter{enumctr}

\newcommand{\R}{\mathbb{R}}
\newcommand{\Z}{\mathbb{Z}}
\newcommand{\C}{\mathbb{C}}

\newcommand{\eps}{\varepsilon}
\renewcommand{\phi}{\varphi}
\newcommand{\rT}{\mathrm {T}}
\begin{document}
\title{\vspace*{-10mm}
Asymptotic stability of  linear fractional systems with constant coefficients and small time dependent perturbations}
\author{
N.D.~Cong\footnote{\tt ndcong@math.ac.vn, \rm Institute of Mathematics, Vietnam Academy of Science and Technology, 18 Hoang Quoc Viet, 10307 Ha Noi, Viet Nam},
T.S.~Doan\footnote{\tt dtson@math.ac.vn, \rm Institute of Mathematics, Vietnam Academy of Science and Technology, 18 Hoang Quoc Viet, 10307 Ha Noi, Viet Nam},
\;and\;
H.T.~Tuan\footnote{\tt httuan@math.ac.vn, \rm Institute of Mathematics, Vietnam Academy of Science and Technology, 18 Hoang Quoc Viet, 10307 Ha Noi, Viet Nam}}
\date{}
\maketitle
\begin{abstract}
Our aim in this paper is to investigate the asymptotic behavior of solutions of the perturbed linear fractional differential system. We show that if the original linear autonomous system is asymptotically stable then under the action of small (either linear or nonlinear) nonautonomous perturbations the trivial solution of the perturbed system is also asymptotically stable. 
\end{abstract}

{\bf Keywords: } {\small Fractional differential equations;
linear systems; stability; asymptotic stability.}

{\it 2010 Mathematics Subject Classification:} {\small 34Dxx, 34A30,
26A33.}

\section{Introduction}
In recent years, fractional differential equations have attracted increasing interest due to
the fact that many mathematical problems in science and engineering can be modeled by fractional differential equations, see e.g.,\ \cite{Podlubny, Kilbas, Kai, Gorenflo_2014}. Although several results on asymptotic behavior of fractional differential equations are already published (e.g., on stability theory \cite{Matignon, Deng_2007, Wen_2008, Sabatier, Li_2011, Qian}, Lyapunov exponents \cite{Cong_1, Cong_2}, attractivity \cite{Chen_2012}, stable manifolds \cite{Cong}), the development of a qualitative theory for fractional differential equations is still in its infancy. One of the reasons for this is the fact that the solution to a fractional differential equation does not generate a semigroup due to the history memory by the induced weakly singular kernel.

In 1996, Matignon \cite{Matignon} studied homogeneous linear fractional differential equations involving Caputo's derivative and has given a well-known stability criterion for these equations. This criterion was developed by several authors. In \cite{Deng_2007}, Deng {\em et al.} studied the stability of some fractional systems with multiple time delays. Later, Sabatier {\em et al.}~\cite{Sabatier} used Linear Matrix Inequality in the stability analysis of inhomogeneous linear fractional systems. In 2010,  Qian {\em et al.}~\cite{Qian}  investigated the stability of fractional differential equations with Riemann--Liouville derivative for linear systems, perturbed systems and time-delayed systems. 

In this paper, we consider the $d$-dimensional fractional differential equation involving the \emph{Caputo's derivative of order $\alpha \in (0,1)$}:
\begin{equation}\label{mainEq0}
^{C\!}D_{0+}^\alpha x(t)=Ax(t)+f(t,x(t))
\end{equation}
with the initial condition
\[
x(0)=x_0,
\]
where $A\in  \R^{d\times d}$ is a  constant ($d\times d$)-matrix and 
$f:[0,\infty)\times \R^d\rightarrow \R^d$ is a continuous vector-valued function such that 
\begin{equation}\label{f0-Cond}
f(t,0)=0\qquad \hbox{for all}\quad t\geq 0,
\end{equation}
and there exists a continuous function $K:[0,\infty)\rightarrow \R_+$ satisfying
\begin{equation}\label{LipCond}
\|f(t,x)-f(t,y)\|\le K(t)\|x-y\| \qquad\hbox{for all}\quad t\geq 0\quad\hbox{and}\quad x,y\in \R^d.
\end{equation}
From \eqref{f0-Cond}--\eqref{LipCond} it follows that the fractional differential equation \eqref{mainEq0} has unique solution for any given initial value (see Baleanu and Mustafa~\cite[Theorem~2]{Baleanu}), and $x\equiv 0$ is the trivial solution of \eqref{mainEq0}.

The case when $f(t,x)$ is linear in $x$ is of special interest and will be treated in one section of the paper; namely we will
 consider the $d$-dimensional fractional differential equation involving the \emph{Caputo's derivative of order $\alpha \in (0,1)$}:
\begin{equation}\label{mainEq}
^{C\!}D_{0+}^\alpha x(t)=[A+Q(t)]x(t)
\end{equation}
with the initial condition
\[
x(0)=x_0,
\]
where $A\in  \R^{d\times d}$ is a  constant ($d\times d$)-matrix and 
$Q:\R_+\rightarrow \R^{d\times d}$ is a continuous matrix-valued function.

If $f$ and $Q$ vanish on $[0,\infty)$, the systems \eqref{mainEq0} and \eqref{mainEq} reduce to the linear time-invariant fractional differential equation
\begin{equation}\label{Eq2}
^{C\!}D_{0+}^\alpha x(t)=Ax(t).
\end{equation}
System \eqref{Eq2} is called the original unperturbed system, whereas \eqref{mainEq} is called the (linear) perturbed system and $Q$ is called the (linear) perturbation, \eqref{mainEq0} is called the (nonlinear) perturbed system and $f$ is called the (nonlinear) perturbation.

We are interested in the asymptotic stability of the trivial solution of \eqref{mainEq0} and \eqref{mainEq}. It is natural to expect that if the original unperturbed system \eqref{Eq2} is asymptotically stable and the perturbations $f$ and $Q$ are small is some sense then the perturbed systems \eqref{mainEq0} and \eqref{mainEq} are asymptotically stable, since such kind of results exist in the theory of ordinary differential equations, see e.g., Coddington and Levinson~\cite[Chapter 13]{Coddington}, Adrianova~\cite[Chapter IV, \S 3]{Adrianova}. In this paper we will show that this is also the case for fractional differential equations. Note that if the unperturbed system \eqref{Eq2} is asymptotically stable and the nonlinear perturbation $f$ having Lipschitz constant uniformly small in a neighborhood of the origin, then the trivial solution of the nonlinear perturbed system \eqref{mainEq0} is also asymptotically stable, see \cite{Cong_3}.

It is well known that the trivial solution of the original unperturbed system \eqref{Eq2}  is asymptotically stable if and only if the spectrum $\sigma(A)$ of the matrix $A\in  \R^{d\times d}$ (spectrum $\sigma(A)$ is the set of eigenvalues of the matrix $A$) satisfies the condition
\begin{equation}\label{StabCond}
\sigma(A)\subset \left\{\lambda \in \C \setminus \{0\}: |\arg(\lambda)|>\dfrac{\alpha \pi}{2}\right\},
\end{equation}
see Diethelm~\cite[Theorem~7.20, p.~158]{Kai}. 

Let us look at the linear perturbed system \eqref{mainEq}, since $Q$ is continuous, for any given initial value the equation \eqref{mainEq} has unique solution existing on the whole $\R_+$ (see Baleanu and Mustafa~\cite[Theorem~2]{Baleanu} and Tisdell~\cite[Theorem~6.4]{Tisdell2012}).
We prove that, provided \eqref{StabCond} is satisfied, if the perturbation $Q$ is small in some sense, the trivial solution of \eqref{mainEq} is asymptotically stable. To do this, we need two preparatory steps. First, using a variation of constants formula which provides the link between the solutions of  the perturbed system \eqref{mainEq} and the solutions of the original unperturbed system \eqref{Eq2}, we define the Lyapunov--Perron operator associated with the equation \eqref{mainEq}, see Theorem \ref{Var_Const_Form_1}. Then, using some properties of the Mittag-Leffler functions and the assumption that $Q$ is small, we estimate this operator. Consequently, the asymptotic stability of \eqref{mainEq} is showed. 

Now, for the nonlinear perturbed system \eqref{mainEq0} we will show that with small modifications the arguments for the linear case will work also for the nonlinear one, hence we will get similar stability theorems for the case of nonlinear perturbations.

The paper is organized as follows: Section~\ref{sec.preliminaries} is a preparatory section where we recall some basic notions and results from fractional calculus and some asymptotic estimations of Mittag-Leffler functions which are needed later for the proofs of our stability theorems. Section~\ref{sec.mainresults} is devoted to the main results on asymptotic stability of the trivial solution of the linear perturbed system \eqref{mainEq} under various assumptions on smallness of $Q$---uniform small $Q$ (Theorem~\ref{stability-by-uniformly-small-pertubation}) and decaying $Q$ (Theorem~\ref{stability-by-decaying-pertubation}).   Section~\ref{sec.general.theory} is devoted to the general case of nonlinear perturbed system \eqref{mainEq0}, where we will formulate and prove stability theorems for \eqref{mainEq0} under various assumptions on smallness of $f$. 

To conclude this introductory section, we introduce some notations which are used throughout the paper.

We denote by $\R_+$ the set of all nonnegative real numbers, by $\Z_+$ the set of all nonnegative integers. Let $\R^d$ be endowed with an arbitrary norm $\|\cdot\|$. Denote by $C([0,\infty);\R^d)$ the space of continuous functions from $[0,\infty)$ to $\R^d$,
 and by $\left(C_\infty(\R^d),\|\cdot\|_\infty\right)\subset C([0,\infty);\R^d)$ the space of all continuous functions $\xi:\R_+\rightarrow \R^d$ 
which are uniformly bounded on $\R_+$, i.e.,
\[
\|\xi\|_\infty:=\sup_{t\in \R_+}\|\xi(t)\|<\infty.
\]
It is well known that $\left(C_\infty(\R^d),\|\cdot\|_\infty\right)$ is a Banach space. 
\section{Preliminaries}\label{sec.preliminaries}
\subsection{Fractional calculus}
We start this subsection by briefly recalling a framework of fractional calculus and fractional differential equations. We refer the reader to the books \cite{Kai,Kilbas} for more details.
Let $\alpha>0$ and $[a,b]\subset \R$. Let $x:[a,b]\rightarrow \R$ be a measurable function such that $x\in L^1([a,b])$, i.e.,\ $\int_a^b|x(\tau)|\;d\tau<\infty$. Then, the Riemann--Liouville integral of order $\alpha$ is defined by
\[
I_{a+}^{\alpha}x(t):=\frac{1}{\Gamma(\alpha)}\int_a^t(t-\tau)^{\alpha-1}x(\tau)\;d\tau\quad \hbox{ for } t\in (a,b],
\]
where the Gamma function $\Gamma:(0,\infty)\rightarrow \R$ is defined as
\[
\Gamma(\alpha):=\int_0^\infty \tau^{\alpha-1}\exp(-\tau)\;d\tau,
\]
see e.g., Diethelm~\cite{Kai}. The corresponding Riemann--Liouville fractional derivative of order $\alpha $ is given by
\[
^{R\!}D_{a+}^\alpha x(t):=(D^m I_{a+}^{m-\alpha}x)(t) \qquad\hbox{ for } t\in (a,b],
\]
where $D=\frac{d}{dt}$ is the usual derivative and $m:=\lceil\alpha\rceil$ is the smallest integer bigger or equal to $\alpha$. On the other hand, the \emph{Caputo fractional derivative } $^{C\!}D_{a+}^\alpha x$ of a function $x\in C^m([a,b])$, which was introduced by Caputo (see e.g., Diethelm~\cite{Kai}), is defined by
\[
^{C\!}D_{a+}^\alpha x(t):=(I_{a+}^{m-\alpha}D^m x)(t),\qquad \hbox{ for } t\in (a,b].
\]
The Caputo fractional derivative of a $d$-dimensional vector function $x(t)=(x_1(t),\dots,x_d(t))^{\rT}$ is defined component-wise as $$^{C\!}D^\alpha_{0+}x(t)=(^{C\!}D^\alpha_{0+}x_1(t),\dots,^{C\!}D^\alpha_{0+}x_d(t))^{\rT}.$$
Let us look at the unperturbed system \eqref{Eq2}. Since $A$ is a constant matrix, this equation can be solved explicitly; namely,
$E_{\alpha}(t^{\alpha}A)x$ solves \eqref{Eq2} with the initial condition $x(0)=x$, where the \emph{Mittag-Leffler matrix function} $E_{\alpha,\beta}(A)$, for $\beta\in\R$ and a matrix $A\in\R^{d\times d}$ is defined as
\[
E_{\alpha,\beta}(A):=\sum_{k=0}^\infty \frac{A^k}{\Gamma(\alpha k+\beta)},\qquad E_{\alpha}(A):=E_{\alpha,1}(A),
\]
see, e.g.,  Bonilla {\em et al.}~\cite{Bonilla2007} and Diethelm~\cite{Kai}. Since $Q$ is time dependent, it is in general impossible to provide an explicit form of the solution of \eqref{mainEq}. However, using the variation of constants formula, see e.g., Bonilla {\em et al.}~\cite{Bonilla2007}, Kilbas {\em et al.}~\cite[Theorem 5.15, p.~323]{Kilbas} and Diethelm~\cite[Theorem 7.2, p.~135]{Kai}, we are able to characterize a solution as a fixed point of the associated Lyapunov--Perron operator: 

For any $x\in \R^{d}$, the operator $\mathcal{T}_{x}: C([0,\infty);\R^d)\rightarrow C([0,\infty);\R^d)$, which  is defined by
\begin{equation}\label{eqn.LPO.nonlinear}
\mathcal{T}_{x}(\xi)(t)=E_\alpha(t^\alpha A)\,x+\int_0^t (t-\tau)^{\alpha-1}E_{\alpha,\alpha}((t-\tau)^\alpha A)f(\tau,\xi(\tau))\;d\tau
\end{equation}
is called the \emph{Lyapunov--Perron operator associated with \eqref{mainEq0}}. (If in \eqref{eqn.LPO.nonlinear} we change $f(\tau,\xi(\tau))$ to $Q(\tau)\xi(\tau)$ then we get 
\emph{Lyapunov--Perron operator associated with \eqref{mainEq}}.)  The role of this operator is stated in the following theorem. The proof of this theorem is a direct consequence of the variation of constants formula and the existence and uniqueness of solutions for initial value problems for \eqref{mainEq0}. 
\begin{theorem}
\label{Var_Const_Form_1}
Let $x\in\R^{d}$ be arbitrary and $\xi:\R_+\rightarrow \R^{d}$ be a continuous function satisfying that $\xi(0)=x$. Then, the following statements are equivalent:
\begin{itemize}
\item [(i)] $\xi$ is a solution of \eqref{mainEq0} satisfying the initial condition $x(0)=x$.
\item [(ii)] $\xi$ is a fixed point of the operator $\mathcal {T}_{x}$.
\end{itemize}
\end{theorem}
\subsection{Mittag-Leffler functions}\label{MLF}
In this subsection, we present some estimations involving the Mittag-Leffler function and its derivative. The results are needed for the proofs of the stability theorems presented in Sections~\ref{sec.mainresults} and \ref{sec.general.theory}.  These results are light refinements and adaption of the known results in the theory of Mittag-Leffler functions to our case. To derive the estimations one uses the spectral representation of Mittag-Leffler functions (see Podlubny~\cite{Podlubny}) and Jordan normal form of matrices (see Lancaster and Tismenetsky~\cite{Lancaster}). To save the length of the paper we do not give full proofs of the theorem, but give only sketch of the proofs. 

\begin{theorem}\label{thm12}
Let $\lambda \in \C\setminus\{0\}$ with $\frac{\alpha \pi }{2}< |\text{arg}(\lambda)|\leq \pi$ and $l\in \Z_+$. Then, there exist positive constants $M_l(\alpha,\lambda), \hat{M}_l(\alpha,\lambda)$ and a positive real number $t_0$ such that the following statements hold
\begin{itemize}
\item [(i)] $|\dfrac{d^l}{d\lambda^l} E_{\alpha}(\lambda t^{\alpha})| \le \frac{M_l(\alpha,\lambda)}{t^{\alpha}}\quad  \hbox{for any } t>t_0,$
\item [(ii)] $|\dfrac{d^l}{d\lambda^l} E_{\alpha,\alpha}(\lambda t^{\alpha})| \le \frac{\hat {M}_l(\alpha,\lambda)}{t^{2\alpha}}\quad  \hbox{for any } t>t_0.$
\end{itemize}
\end{theorem}
For a proof of this theorem one uses integral representation of Mittag-Leffler functions and method of estimations of the integrals similar to that of the proof of Theorem 1.3 and  Theorem 1.4 in the book by Podlubny~\cite[pp. 32--34]{Podlubny}. Note that the case $l=0$ of our Theorem~\ref{thm12} is contained in the conclusion of Theorem 1.4 by Podlubny~\cite[Formula (1.143), p. 34]{Podlubny}.  
\begin{theorem}\label{thm.spec.estimate}
Let $A\in \R^{d\times d}$. Assume that the spectrum of $A$ satisfies the relation 
$$
\sigma(A)\subset\left\{\lambda\in \C\setminus\{0\}:|\arg(\lambda)|>\frac{\alpha \pi}{2}\right\}. 
$$
Then, the following statements hold:
\begin{itemize}
\item [(i)]  $\lim_{t\to \infty}\|E_\alpha(t^\alpha A)\|=0;$
\item [(ii)]  $\int_0^\infty \tau^{\alpha -1}\| E_{\alpha,\alpha}(\tau^\alpha A)\| \,d\tau <\infty.$
\end{itemize}
\end{theorem}
For a proof of this theorem one uses the estimations given in Theorem~\ref{thm12}, the series presentation of Mittag-Leffler function of matrix argument, and the Jordan normal form of matrices (see, e.g.,  
Lancaster and Tismenetsky~\cite{Lancaster} for the Jordan normal form of matrices and evaluation of series of Jordan matrices).

\section{Asymptotic stability of linear FDEs with constant coefficients and small linear time dependent perturbations}\label{sec.mainresults}

In this section, we consider the linear system \eqref{mainEq}, i.e., the following system
\begin{equation*}
^{C\!}D_{0+}^{\alpha}x(t)=[A+Q(t)]x(t),
\end{equation*}
where $A\in \R^{d\times d}$, $Q:[0,\infty)\to \R^{d\times d}$ is a continuous matrix-valued function. In what follows, we denote the solution of \eqref{mainEq} with the initial condition $x(0)=x_0$ by $\varphi(\cdot,x_0)$.  We recall below the notions of stability and asymptotic stability of the trivial solution of \eqref{mainEq}, cf.\ Diethelm~\cite[Definition 7.2, p.~157]{Kai}.
\begin{definition}\label{DS}
The trivial solution of \eqref{mainEq} is called \emph{stable} if for any $\varepsilon >0$ there exists $\delta=\delta(\varepsilon)>0$ such that for every $\|x_0\|<\delta$ we have
\[
\|\phi(t,x_0)\|\leq \eps\qquad\hbox{for } t\ge 0.
\]
The trivial solution is called \emph{asymptotically stable} if it is stable and there exists $\widehat{\delta}> 0$ such that $\lim_{t\to \infty}\phi(t,x_0)=0$ whenever $\|x_0\|<\widehat\delta$.
\end{definition}

Now we will state and prove our first stability result for linear fractional differential equations.

\begin{theorem}[Robust Stability]\label{MainTheorem}
Assume that the spectrum of the matrix $A$ satisfies the condition
$$
\sigma(A)\subset\left\{\lambda\in \C \setminus \{0\} :|\arg(\lambda)|>\frac{\alpha \pi}{2}\right\},
$$
and, in addition, $Q$ satisfies
\begin{equation}\label{mainCond}
q:=\sup_{t\ge 0} \int_0^t  (t-\tau)^{\alpha-1}\|E_{\alpha,\alpha}((t-\tau)^\alpha A) Q(\tau)\|\;d\tau <1.
\end{equation}
Then the trivial solution of \eqref{mainEq} is asymptotically stable.
\end{theorem}
\begin{proof}
We follow the lines of the proof of Theorem 5 in [9] with some modifications to adapt to our case. Let $\eps>0$ be arbitrary. By virtue of Theorem \ref{thm.spec.estimate}(i),  $\sup_{t\geq 0} \|E_\alpha(t^\alpha A)\|\in (1,\infty)$. Therefore, 
\[
\delta:=\frac{(1-q)\eps}{\sup_{t\ge 0}\|E_\alpha(t^\alpha A)\|} \in (0,\eps),
\]
where $q$ is defined as in \eqref{mainCond}. To prove the asymptotic stability of the trivial solution of \eqref{mainEq}, it is sufficient to show that if $\|x\|\leq \delta$ then $\phi(\cdot,x)\in B_{C_{\infty}}(0,\eps)$ and $\lim_{t\to\infty}\phi(t,x)=0$, where 
\[
B_{C_{\infty}}(0,\eps):=\{\xi\in C([0,\infty);\R^d):\left||\xi|\right|_\infty\le \eps\}\subset C_\infty(\R^d)\subset
C([0,\infty);\R^d).
\]
Choose and fix an arbitrary $x\in \R^d$ such that $\|x\|\le \delta$. Let $\mathcal{T}_x$ be the Lyapunov--Perron operator associated with \eqref{mainEq}.  For $\xi \in B_{C_{\infty}}(0,\eps)$, we have
\begin{align*}
\|\mathcal{T}_x(\xi)(t)\|&\le \|E_\alpha(t^\alpha A)x\|+\int_0^t (t-u)^{\alpha-1}\|E_{\alpha,\alpha}((t-\tau)^\alpha A))Q(\tau)\xi(\tau)\|\;d\tau \\
& \le \delta \sup_{t\ge 0}\|E_\alpha(t^\alpha A)\|+\eps q\\
&\le \eps. 
\end{align*}
Consequently, $\mathcal{T}_x(B_{C_{\infty}}(0,\eps))\subset B_{C_{\infty}}(0,\eps)$. Moreover, for any $\xi,\tilde{\xi} \in B_{C_{\infty}}(0,\eps)$, we have
\begin{align*}
\mathcal{T}_x(\xi)(t)-\mathcal{T}_x(\tilde{\xi})(t)=\int_0^t (t-\tau)^{\alpha-1} E_{\alpha,\alpha}((t-\tau)^\alpha A)Q(\tau)[\xi(\tau)-\tilde{\xi}(\tau)]\;d\tau.
\end{align*}
Hence,
\begin{align*}
\|\mathcal{T}_x(\xi)-\mathcal{T}_x(\tilde{\xi})\|_\infty &\le \|\xi-\tilde{\xi}\|_\infty \sup_{t\ge 0}\int_0^t(t-\tau)^{\alpha-1}\| E_{\alpha,\alpha}((t-\tau)^\alpha A)Q(\tau)\|\;d\tau\\
&\le q \|\xi-\tilde{\xi}\|_\infty,
\end{align*}
and $\mathcal{T}_x$ is contractive if restricted to the closed ball $B_{C_\infty}(0,\eps)$. Using the Contraction Mapping Principle, there exists a unique fixed point $\xi \in B_{C_\infty}(0,\eps)$ of $\mathcal{T}_x$. According to Theorem \ref{Var_Const_Form_1}, this point is also the unique solution of \eqref{mainEq} satisfying the initial condition $x(0)=x$, i.e., $\phi(t,x)=\xi(t)$ for $t\geq 0$. Hence, $|\phi(t,x)|\leq \eps$ for $t\geq0$. To conclude the proof, we need to show that $a:=\limsup_{t\to\infty}\|\xi(t)\|=0$. Suppose the contrary that $a>0$. Then, there exists $T>0$ such that
\[
\|\xi(t)\|\le a+\frac{1-q}{2q+1}a \qquad \textup{for any}\quad t\ge T.
\]
According to Theorem~\ref{thm.spec.estimate}(ii), we have
\begin{eqnarray*}
&&\limsup_{t\to\infty}\left\| \int_0^{T}(t-\tau)^{\alpha-1}E_{\alpha,\alpha}((t-\tau)^\alpha A)Q(\tau)\xi(\tau)\,d\tau\right\|\\[1.5ex]
&\le& 
\eps \max_{t\in [0,T]}\|Q(t)\|\limsup_{t\to \infty}\int_{t-T}^t \tau^{\alpha-1}\|E_{\alpha,\alpha}(\tau^\alpha A)\|\, d\tau\\
&= &0.
\end{eqnarray*}
Therefore, using equality $\xi=\mathcal{T}_x\xi$ and Theorem~\ref{thm.spec.estimate}(i), we obtain that 
\begin{align*}
a =&\limsup_{t\to\infty}
\left\|\int_{T}^t(t-\tau)^{\alpha-1}E_{\alpha,\alpha}((t-\tau)^\alpha A)Q(\tau)\xi(\tau)d\tau
\right\|\\
\le&\left(a+\frac{1-q}{2q+1}a\right)\sup_{t\geq T}\int_T^t (t-\tau)^{\alpha-1}\|E_{\alpha,\alpha}((t-\tau)^\alpha A)Q(\tau)\| d\tau \\
\leq&\; \frac{a(2+q)}{2q+1}q \; = \; a\frac{2q + q^2}{2q+1} \;< \;a,
\end{align*} 
which is a contradiction. Hence, $a=0$ and the proof is complete.
\end{proof}
\begin{theorem}[Stability by uniformly small perturbation]\label{stability-by-uniformly-small-pertubation}
Assume that the spectrum of the matrix $A$ satisfies the condition
$$
\sigma(A)\subset \left\{\lambda\in \C \setminus \{0\} :|\arg(\lambda)|>\frac{\alpha \pi}{2}\right\}.
$$ 
Then there exists a positive number $\varepsilon>0$ such that if $Q$ satisfies
\begin{equation}\label{uniformCond}
\sup_{t\ge 0} \| Q(t)\|\ < \varepsilon,
\end{equation}
the trivial solution of \eqref{mainEq} is asymptotically stable.
\end{theorem}
\begin{proof}
By virtue of Theorem~\ref{thm.spec.estimate}(ii), we can choose
$$
0< \varepsilon := \frac{1}{2 \int_0^\infty u^{\alpha -1}\| E_{\alpha,\alpha}(u^\alpha A)\| \,du}<\infty.
$$
Clearly, if $Q$ satisfies \eqref{uniformCond} then the condition \eqref{mainCond} holds, hence our theorem follows from
 Theorem~\ref{MainTheorem}.
\end{proof}

Before going to the theorem on stability of the linear system \eqref{mainEq}  in case of decaying $Q$ we need the following auxiliary result which is of independent interest.
\begin{theorem}[Lyapunov stability of finite dimensional linear FDE]\label{thm.d-dim}
Consider a $d$-dimensional linear fractional differential equation on $\R_+$:
\begin{equation}\label{eqn.d-dim}
^{C\!}D_{0+}^\alpha x(t)=B(t)x(t),
\end{equation}
where $B:\R_+\rightarrow \R^{d\times d}$ is a continuous matrix-valued function. 
Then the following statements are equivalent:
\begin{itemize}
\item [(i)] The trivial solution of the equation \eqref{eqn.d-dim} is stable;
\item [(ii)] Any solution of \eqref{eqn.d-dim} is bounded on $\R_+$;
\item[(iii)] There exist $d$ linearly independent initial vectors $x_1,x_2,\ldots,x_d\in\R^d$ such that the solutions of \eqref{eqn.d-dim} starting at time 0 at those vectors are bounded on $\R_+$.
\end{itemize}
\end{theorem}
\begin{proof}
First we note that since $B$ is continuous on $\R_+$ the initial value problem for \eqref{eqn.d-dim} has unique solution existing on the whole $\R_+$ for any given initial value (see Baleanu and Mustafa~\cite[Theorem~2]{Baleanu} and Tisdell~\cite[Theorem~6.4]{Tisdell2012}). The stability of the trivial solution of \eqref{eqn.d-dim} is defined according to Definition~\ref{DS}; this stability is also called {\em Lyapunov stability}.
Due to the linearity of the fractional Caputo differentiation and linearity of \eqref{eqn.d-dim} there is a bijection between the solution space of \eqref{eqn.d-dim} and the vector space $\R^d$ of initial values of  \eqref{eqn.d-dim}.

$(i) \Rightarrow (ii)$: If the trivial solution of \eqref{eqn.d-dim} is stable then any solution started from a suitably small ball around origin must be bounded on $\R_+$. Then (ii) follows by linearity.

$(ii) \Rightarrow (iii)$: Obvious.

$(iii) \Rightarrow (i)$: Assume that $x_1,x_2,\ldots,x_d\in\R^d$ is a basis of $\R^d$. 
For brevity, let  $x_1(t),x_2(t),\ldots,x_d(t)$ denote the solutions of \eqref{eqn.d-dim} starting at time 0 at $x_1,x_2,\ldots,x_d$, respectively. From boundedness of $x_1(t),\dots, x_d(t)$, we have 
\begin{equation}\label{New_Eq1a}
M:=\max_{1\leq i\leq d} \sup_{t\geq 0} \|x_i(t)\| < \infty.
\end{equation}
Define $\mathcal S:= \big\{(c_1,\dots,c_d)\in [-1,1]^d: \max_{1\leq i \leq d} |c_i|=1\big\}$ and a continuous map $\pi: \mathcal S\rightarrow \R$ by 
\[
\pi(c_1,\dots,c_d):=\|c_1x_1 + c_2x_2 + \cdots + c_dx_d\|.
\]
Since $x_1,\dots,x_d$ is a basis of $\R^d$ and $\mathcal S$ is a compact set it follows that
\begin{equation}\label{New_Eq1b}
m:=\min_{(c_1,\dots,c_d)\in \mathcal S}\pi(c_1,\dots,c_d)>0.
\end{equation}
To prove stability of the trivial solution, let $\eps>0$ be arbitrary. Set $\delta:=\frac{m\eps}{2dM}$. Let $x\in \R^d\setminus\{0\}$ be an arbitrary non-zero vector such that $\|x\|\leq \delta$ and $x(t)$ denote the solution of \eqref{eqn.d-dim} starting at time 0 at $x$. The vector $x$ is represented uniquely by $x=\sum_{i=1}^d\alpha _i x_i$. By linearity, we have $x(t)=\sum_{i=1}^d\alpha _i x_i(t)$. Hence, from \eqref{New_Eq1a} we have 
\[
\|x(t)\|
\leq 
\sum_{i=1}^d |\alpha_i| M 
\leq d M \max_{1\leq i\leq d} |\alpha_i|
\quad\hbox{for all } t\in \R_{+}.
\]
On the other hand, by \eqref{New_Eq1b} we have 
\[
\pi\left(\frac{\alpha_1}{\max_{1\leq i\leq d} |\alpha_i|},\dots,\frac{\alpha_d}{\max_{1\leq i\leq d} |\alpha_i|}\right)
=
\frac{\|x\|}{\max_{1\leq i\leq d} |\alpha_i|}
\geq 
m,
\]
which implies that $\max_{1\leq i\leq d} |\alpha_i|\leq \frac{\delta}{m}$. Consequently, 
\[
\|x(t)\|
\leq 
dM \frac{\delta}{m}=\frac{\eps}{2}\quad \hbox{for all } t\in \R_{+},
\]
which completes the proof.
\end{proof}
\begin{theorem}[Stability by decaying perturbation]\label{stability-by-decaying-pertubation}
Assume that the spectrum of the matrix $A$ satisfies the condition
$$
\sigma(A)\subset\left\{\lambda\in \C \setminus \{0\} :|\arg(\lambda)|>\frac{\alpha \pi}{2}\right\}.
$$ 
If the matrix $Q$ is decaying to zero, i.e., 
\begin{equation}\label{decayCond}
\lim_{t\rightarrow\infty} \|Q(t)\|=0,
\end{equation}
then the trivial solution of \eqref{mainEq} is asymptotically stable.
\end{theorem}
\begin{proof}
Fix an arbitrary $x\in\R^d$.
First, we show that any solution of \eqref{mainEq} is bounded. To this end we equip the space $C_\infty(\R^d)$ of bounded continuous vector-functions with a new norm $\|\cdot\|_\beta$ which is equivalent to the norm $\|\cdot\|_\infty$ so that  $(C_\infty(\R^d),\|\cdot\|_\beta)$ is a new Banach space, in which the Lyapunov--Perron operator  associated with \eqref{mainEq} is a contraction. Note that since $Q$ is continuous and decaying it is uniformly bounded, hence by virtue of Theorem~\ref{thm.spec.estimate} we can find a constant $M>1$ such that
\begin{equation}\label{eqn.big-M}
\begin{array}{cll}
\sup_{t\geq 0} \|E_{\alpha,\alpha}(t^\alpha A)\|\times \sup_{t\geq 0} \|Q(t)\| &\le & \frac{M}{\Gamma(\alpha)},\\[1,5ex]
\sup_{t\geq 0} \int_0^t \tau^{\alpha-1}\|E_{\alpha,\alpha}(\tau^\alpha A)\|\, d\tau& \le & M.\\
\end{array}
\end{equation}
By (\ref{decayCond}) we can find $T>0$ such that
\begin{equation}\label{eqn.tail}
\sup_{t\geq T} \|Q(t)\| < \frac{1}{5M}.
\end{equation}
We introduce a function $\beta(\cdot): \R^+\rightarrow \R^+$ by the formula
$$
\beta(t) := \left\{ \begin{array}{ll}
E_\alpha(5Mt^\alpha) &\hbox{if}\;  0 \leq t \leq T,\\[1,5ex]
E_\alpha(5MT^\alpha)  & \hbox{if}\; t\geq T,
\end{array}
\right.
$$
and define a norm $\|\cdot\|_\beta$ in the space  $C_\infty(\R^d)$ of bounded continuous vector-functions
by setting $\|y\|_\beta :=  \sup_{t\geq 0}\frac{\|y(t)\|}{\beta(t)}$ for any $y\in C_\infty(\R^d)$. This norm is equivalent to the   sup norm $\|\cdot\|_\infty$ because
$$
\frac{1}{\beta(T)} \|y\|_\infty 
\leq 
\|y\|_{\beta}
\leq \|y\|_\infty\quad \hbox{for all } y \in C_{\infty}(\R^d).
$$
Thus, the space $(C_\infty(\R^d), \|\cdot\|_\beta)$ is a Banach  space.
 Now, fix an arbitrary $x\in\R^d$,  we show that the Lyapunov--Perron operator  associated with \eqref{mainEq} defined as in \eqref{eqn.LPO.nonlinear} is a contraction in 
$(C_\infty(\R^d), \|\cdot\|_\beta)$.
For any $\xi\in C_\infty(\R^d)$, due to the assumptions of the theorem, taking into account Theorem~\ref{thm.spec.estimate} we have
$$
(\mathcal{T}_{x}\xi)(t)=E_\alpha(t^\alpha A)\,x+\int_0^t (t-\tau)^{\alpha-1}E_{\alpha,\alpha}((t-\tau)^\alpha A)Q(\tau)\xi(\tau)\;d\tau\in C_\infty(\R^d),
$$
hence $\mathcal{T}_{x}$ is a self map of $C_\infty(\R^d)$.
Now, for any $\xi,\tilde\xi\in C_\infty(\R^d)$ we have
$$
\mathcal{T}_x(\xi)(t)-\mathcal{T}_x(\tilde{\xi})(t)=\int_0^t (t-\tau)^{\alpha-1} E_{\alpha,\alpha}((t-\tau)^\alpha A)Q(\tau)[\xi(\tau)-\tilde{\xi}(\tau)]\;d\tau.
$$
For $0\leq t\leq T$,  by \eqref{eqn.big-M} we have 
\begin{eqnarray}
\frac{\|\mathcal{T}_x(\xi)(t)-\mathcal{T}_x(\tilde{\xi})(t)\|}{\beta(t)}
&\leq&\frac{\|\xi-\tilde \xi\|_{\beta}}{\beta(t)}\frac{M}{\Gamma(\alpha)}
\int_0^t (t-\tau)^{\alpha-1} \beta(\tau)\;d\tau\nonumber\\
&=&\frac{\|\xi-\tilde \xi\|_{\beta}}{E_\alpha(5Mt^\alpha)}\frac{M}{\Gamma(\alpha)}
\int_0^t (t-\tau)^{\alpha-1} E_\alpha(5M\tau^\alpha)\;d\tau\nonumber\\
&\leq&\frac{M}{5M}\|\xi-\tilde \xi\|_{\beta}=\frac{\|\xi-\tilde \xi\|_{\beta}}{5}.\label{eqn.tail1}
\end{eqnarray}
For $t>T$ we have
\begin{eqnarray*}
\frac{\|\mathcal{T}_x(\xi)(t)-\mathcal{T}_x(\tilde{\xi})(t)\|}{\beta(t)}
&=& \frac{\|\mathcal{T}_x(\xi)(t)-\mathcal{T}_x(\tilde{\xi})(t)\|}{E_\alpha(5MT^\alpha)}\\
&&\hspace*{-3cm} = \; \frac{1}{E_\alpha(5MT^\alpha)}\|\int_0^t (t-\tau)^{\alpha-1}E_{\alpha,\alpha}((t-\tau)^\alpha A)Q(\tau)(\xi(\tau)-\tilde\xi(\tau))\;d\tau\|\\
&&\hspace*{-3cm} \leq \;
\frac{1}{E_\alpha(5MT^\alpha)}\|\int_0^T (t-\tau)^{\alpha-1}E_{\alpha,\alpha}((t-\tau)^\alpha A)Q(\tau)(\xi(\tau)-\tilde\xi(\tau))\;d\tau\|\\
&&\hspace*{-2.5cm} +\; \frac{1}{E_\alpha(5MT^\alpha)}\|\int_T^t (t-\tau)^{\alpha-1}E_{\alpha,\alpha}((t-\tau)^\alpha A)Q(\tau)(\xi(\tau)-\tilde\xi(\tau))\;d\tau\|.
\end{eqnarray*}
Therefore, using \eqref{eqn.big-M} and \eqref{eqn.tail} we get
\begin{eqnarray*}
\frac{\|\mathcal{T}_x(\xi)(t)-\mathcal{T}_x(\tilde{\xi})(t)\|}{\beta(t)}
&\leq& \frac{M}{E_\alpha(5MT^\alpha)}\frac{1}{\Gamma(\alpha)}\int_0^T (t-\tau)^{\alpha-1}\|\xi(\tau)-\tilde\xi(\tau)\|\;d\tau\\
&&+\;\|\xi-\tilde \xi\|_{\beta}\frac{1}{5M}
\int_T^t (t-\tau)^{\alpha-1}\|E_{\alpha,\alpha}((t-\tau)^\alpha A)\|\;d\tau\\
&\leq& \frac{M}{E_\alpha(5MT^\alpha)} \frac{1}{\Gamma(\alpha)}\int_0^T (T-\tau)^{\alpha-1}\|\xi(\tau)-\tilde\xi(\tau)\|\;d\tau\\
&&+\;\|\xi-\tilde \xi\|_{\beta}\frac{1}{5M}\times M.
\end{eqnarray*}
Consequently, for all $t\geq T$ we have
\begin{equation}\label{eqn.tail2}
\frac{\|\mathcal{T}_x(\xi)(t)-\mathcal{T}_x(\tilde{\xi})(t)\|}{\beta(t)}
\leq M\times \frac{1}{5M} \|\xi-\tilde \xi\|_{\beta} + \frac{\|\xi-\tilde \xi\|_{\beta}}{5}
\leq \frac{1}{2}\|\xi-\tilde \xi\|_{\beta}.
\end{equation}
Combining \eqref{eqn.tail1} with \eqref{eqn.tail2}, we get 
for all $t\geq 0$ the inequality
$$
\frac{\|\mathcal{T}_x(\xi)(t)-\mathcal{T}_x(\tilde{\xi})(t)\|}{\beta(t)}
\leq \frac{1}{2}\|\xi-\tilde \xi\|_{\beta}.
$$
Hence,
\begin{equation}\label{eqn.new.linear}
\|\mathcal{T}_x(\xi)-\mathcal{T}_x(\tilde{\xi})\|_{\beta}
 \leq
  \frac{1}{2}\|\xi-\tilde \xi\|_{\beta},
\end{equation}
what shows that $\mathcal{T}_x$ is a contraction of the Banach space $(C_\infty(\R^d), \|\cdot\|_\beta)$. Consequently, 
$\mathcal{T}_x$ has an unique fixed point in $C_\infty(\R^d)$ which is the unique bounded solution of 
\eqref{mainEq} starting from the initial value $x\in\R^d$. Thus we have shown that any solution of \eqref{mainEq}  is bounded. Therefore, by virtue of Theorem~\ref{thm.d-dim} the trivial solution of \eqref{mainEq} is Lyapunov stable.

Next we show that any solution of \eqref{mainEq} tends to zero. This can be done by using arguments similar to that of the second part of the proof of Theorem~\ref{MainTheorem}.
Consequently, the trivial solution of \eqref{mainEq}  is asymptotically stable.
\end{proof}
\begin{remark}
A closer inspection of the proof of Theorem~\ref{stability-by-decaying-pertubation} shows that for the asymptotic behavior of $Q$ at infinity we only need \eqref{eqn.tail}, hence actually instead the condition \eqref{decayCond} in the formulation of Theorem~\ref{stability-by-decaying-pertubation}  we may only require $\lim_{t\rightarrow\infty} \| Q(t)\|$ be less than a positive number depending on $A$ and $\sup_{t\geq 0} \|Q(t)\|$, i.e., a weaker condition than \eqref{decayCond}.
\end{remark}
%
%
\begin{remark}
In the paper \cite{Qian}, the authors investigated the stability of the linear fractional system with Riemann--Liouville derivative and Caputo derivative, similar to the systems treated in our paper. 
Using a Gronwall's type inequality they obtained some results on asymptotic stability of the trivial solution of the perturbed linear system under some assumptions on the spectrum of the original constant matrix and boundedness of the linear perturbation (see Theorem 4.1(a) and Remark 4.1(a)). 
 Unfortunately, their proof contains some mistakes with application of the Gronwall's inequality  (see \cite[line -7, page 869]{Qian}). This leads to the fact that the statements in Theorem 4.1(a) and Remark 4.1(a) of \cite{Qian} are false. For a counterexample let us consider a scalar fractional differential equation involving Riemann--Liouville derivative of order $\alpha \in (0,1)$ as below:
\begin{equation}\label{ex}
^{R\!}D_{0+}^\alpha x(t)=[-\lambda+b(t)]x(t),
\end{equation}
with the initial condition
\[
\lim_{t\to 0}{}^{R\!}D_{0+}^{\alpha-1} x(t)=x_0.
\]
Assume that $\lambda >0$ and $b(t)\equiv 2\lambda$ on the half line $[0,\infty)$. It is well known that the solution of \eqref{ex} on $(0,\infty)$ is $ t^{\alpha-1} E_{\alpha,\alpha}(\lambda t^\alpha)\, x_0$ (see Podlubny~\cite[Example~4.3, p.~140]{Podlubny}). Since $\lambda>0$, due to the asymptotic behavior of the Mittag-Leffler function $E_{\alpha,\alpha}(\lambda t^\alpha)$ at the infinity, in case $x_0\not= 0$ the solution tends to the infinity as $t$ tends to the infinity. This shows that Theorem 4.1(a) of \cite{Qian} is false. Similarly, Remark 4.1(a) of \cite{Qian} is also false.
\end{remark}

\section{General theory of asymptotic stability of linear FDEs with constant coefficients and small nonlinear time dependent perturbations}\label{sec.general.theory}
 In this section we investigate the asymptotic stability of the nonlinear perturbed system \eqref{mainEq0} with conditions \eqref{f0-Cond} and \eqref{LipCond}, i.e., the equation
$$
^{C\!}D_{0+}^\alpha x(t)=Ax(t)+f(t,x(t))
$$
with the initial condition
\[
x(0)=x_0,
\]
where $A\in  \R^{d\times d}$ is a  constant ($d\times d$)-matrix and 
$f:[0,\infty)\times \R^d\rightarrow \R^d$ is a continuous vector-valued function such that 
$$
f(t,0)=0\qquad \hbox{for all}\quad t\geq 0,
$$
and there exists a continuous function $K:[0,\infty)\rightarrow \R_+$ satisfying
$$
\|f(t,x)-f(t,y)\|\le K(t)\|x-y\| \qquad\hbox{for all}\quad t\geq 0\quad\hbox{and}\quad x,y\in \R^d.
$$
We will show that the results for the linear case presented in Section~\ref{sec.mainresults} can be easily generalized to the general nonlinear case of this section.
Recall that the definition of stability and asymptotic stability of the trivial solution of \eqref{mainEq0} is just the same as the definition for the linear case given in 
Definition~\ref{DS}.

\begin{theorem}[Robust Stability for Nonlinear Equation]\label{MainTheorem.nonlinear}
Assume that the spectrum of the matrix $A$ satisfies the condition
$$
\sigma(A)\subset\left\{\lambda\in \C \setminus \{0\} :|\arg(\lambda)|>\frac{\alpha \pi}{2}\right\},
$$
and, in addition, $K(\cdot)$ satisfies
\begin{equation}\label{mainCond.nonlinear}
q:=\sup_{t\ge 0} \int_0^t  (t-\tau)^{\alpha-1}\|E_{\alpha,\alpha}((t-\tau)^\alpha A) \|K(\tau)\;d\tau <1.
\end{equation}
Then the trivial solution of \eqref{mainEq0} is asymptotically stable.
\end{theorem}
\begin{proof}
We need to just make some obvious changes to the proof of Theorem~\ref{MainTheorem} to get a proof of this theorem.
\end{proof}

\begin{theorem}[Stability by uniformly small Lipschitz perturbation]\label{stability-by-uniformly-small-lipschitz-pertubation}
Assume that the spectrum of the matrix $A$ satisfying
$$
\sigma(A)\subset \left\{\lambda\in \C \setminus \{0\} :|\arg(\lambda)|>\frac{\alpha \pi}{2}\right\}.
$$ 
Then there exists a positive number $\varepsilon>0$ such that if $K$ satisfies
\begin{equation}\label{uniformCond.lip}
\sup_{t\ge 0}  K(t) < \varepsilon,
\end{equation}
the trivial solution of \eqref{mainEq0} is asymptotically stable.
\end{theorem}
\begin{proof}
By a suitable choice of $\varepsilon$ as in the proof of Theorem \ref{stability-by-uniformly-small-pertubation}, from \eqref{uniformCond.lip} we get \eqref{mainCond.nonlinear}, and Theorem~\ref{MainTheorem.nonlinear} is applicable.
\end{proof}
 
To conclude the section we formulate and prove a theorem on asymptotic stability of \eqref{mainEq0} under the condition of decaying Lipschitz constant.
\begin{theorem}[Stability by decaying Lipschitz perturbation]\label{stability-by-decaying-lipschitz-pertubation}
Assume that the spectrum of the matrix $A$ satisfies the condition
$$
\sigma(A)\subset\left\{\lambda\in \C \setminus \{0\} :|\arg(\lambda)|>\frac{\alpha \pi}{2}\right\}.
$$ 
If $K$ is decaying to zero, i.e., 
\begin{equation}\label{decayCond.lip}
\lim_{t\rightarrow\infty} K(t) = 0,
\end{equation}
then the trivial solution of \eqref{mainEq0} is asymptotically stable.
\end{theorem}
\begin{proof}
First we note that the similar Theorem~\ref{stability-by-decaying-pertubation} for linear case was proved with the use of linearity of \eqref{mainEq}, which, in general, is not available in our case of this theorem. 
To overcome the lack of linearity, we do as follows.
We repeat  the proof of Theorem~\ref{stability-by-decaying-pertubation} with obvious changes from $Q(t)x$ to $f(t,x)$, and  $\|Q\|$ replaced by $K$ in the estimations \eqref{eqn.big-M} and \eqref{eqn.tail}. Define the norm $\|\cdot\|_{\beta}$ as in the proof of Theorem~\ref{stability-by-decaying-pertubation}, and let ${\mathcal T}_x$ be the Lyapunov--Perron operator associated with \eqref{mainEq0}. Follow the lines of the proof of Theorem~\ref{stability-by-decaying-pertubation}, similar to  \eqref{eqn.new.linear},  for all $x\in\R^d$, $\xi,\tilde\xi\in C_\infty(\R^d)$ we have
\begin{equation}\label{eqn.new.nonlinear}
\|\mathcal{T}_x(\xi)- \mathcal{T}_x(\tilde{\xi})\|_{\beta}
 \leq \frac{1}{2}
 \|\xi-\tilde \xi\|_{\beta}.
\end{equation}
Let $r>0$ be arbitrary.  Set $B_{C_\infty,\|\cdot\|_\beta}(0,r):=\{\xi\in C([0,\infty);\R^d):\|\xi\|_\beta\le r\}$.
Let $x\in\R^d$ be any vector satisfying the condition
\begin{equation}\label{eqn.nonlinear1}
\|x\| \leq \frac{r}{2\sup_{t\geq 0}{\|E_\alpha(t^\alpha A)\|}} =:r^*.
\end{equation}
Substituting $\tilde\xi\equiv 0$ into \eqref{eqn.new.nonlinear}, then since ${\mathcal T}_x\tilde\xi = E_\alpha(t^\alpha A) x$, taking into account the definition of the norm $\|\cdot\|_\beta$ we get
$$
\|\mathcal{T}_x\xi\|_\beta\le \|x\|\sup_{t\ge 0}\|E_\alpha(t^\alpha A)\|+\frac{r}{2} < r.
$$
Therefore, the Lyapunov--Perron operator \eqref{eqn.LPO.nonlinear} associated with \eqref{mainEq0} with the initial value $x$ satisfying $\|x\|\leq r^*$ is a self map, and together with \eqref{eqn.new.nonlinear} is a contraction, in the closed ball 
$B_{C_\infty,\|\cdot\|_\beta}(0,r)$
of $(C_\infty(\R^d), \|\cdot\|_\beta)$. Hence, since the norm $\|\cdot\|_{\beta}$ and the max norm are equivalent, this shows that \eqref{mainEq0} is stable. 
Thus we proved stability of \eqref{mainEq0} avoiding 
necessarity of using linearity as done in the proof of Theorem~\ref{stability-by-decaying-pertubation}. The proof of asymptotic stability is similar to that of the proof of Theorem~\ref{MainTheorem}.
\end{proof}

\section*{Acknowledgement}
This research of the authors is funded by the Vietnam National Foundation for
Science and Technology Development (NAFOSTED) under Grant Number 101.03-2014.42.


\begin{thebibliography}{1}

\bibitem{Adrianova}
L.Ya.~Adrianova.
\newblock{\em Introduction to Linear Systems of Differential Equations.}
\newblock{Translations of Mathematical Monographs \bf{46}.}
\newblock{Americal Mathematical Society,} Providence, Rhode Island, 1995.
\bibitem{Baleanu}
D.~Baleanu and O.~Mustafa.
\newblock{On the global existence of solutions to a class of fractional
differential equations.}
\newblock{\em Computers and Mathematics with Applications,} {\bf 59} (2010), 1835--1841.
%
\bibitem{Bonilla2007}
B.~Bonilla, M.~Rivero and J.J.~Trujillo.
\newblock{On systems of linear fractional differential equations with constant coefficients.}
\newblock{\em Applied Mathematics and Computation,} {\bf 187}(2007), 68--78.

\bibitem{Chen_2012}
F.~Chen, J.~Nieto and Y.~Zhou.
\newblock{Global attracting for nonlinear fractional differential equations.}
\newblock{\em Nonlinear Analysis: Real World Applications,} {\bf 13}(2012), 287--298.
%
\bibitem{Coddington}
E.A.~Coddington, N.~Levinson.
\newblock{\em Theory of Differential Equations.}
\newblock{McCrow--Hill,} New York, 1955.
\bibitem{Cong}
N.D.~Cong, T.S.~Doan, S.~Siegmund and H.T.~Tuan.
\newblock{On stable manifolds for planar fractional differential equations.}
\newblock{\em Applied mathematics and Computation,} {\bf 226}(2014), 157-168.
\bibitem{Cong_1}
N.D.~Cong, T.S.~Doan and  H.T.~Tuan.
\newblock{On fractional Lyapunov exponent for solutions of Linear fractional differential equations.}
\newblock{\em Fract. Calc. Appl. Anal.,} {\bf 17}(2014), No 2, 285--306.
\bibitem{Cong_2}
N.D.~Cong, T.S.~Doan, H.T.~Tuan and S.~Siegmund.
\newblock{ Structure of the Fractional Lyapunov Spectrum for Linear Fractional Differential Equations.}
\newblock{\em Advances in Dynamical Systems and Applications,} {\bf 9}(2014), 149-159.
%
\bibitem{Cong_3}
N.D.~Cong, T.S.~Doan, S.~Siegmund and H.T.~Tuan.
\newblock{Linearized Asymptotic Stability for Fractional Differential Equations,} {\em arXiv:1512.04989v1.}
%
\bibitem{Deng_2007}
W.H.~Deng, C.P.~Li and J.~H. Lu.
\newblock{Stability analysis of linear fractional differential systems with multiple time delays.}
\newblock{\em Nonlinear Dynamics,} {\bf 48}(2007), 409--416.
%

\bibitem{Kai}
K.~Diethelm.
\newblock{The analysis of fractional differential equations. An application-oriented exposition using differential operators of Caputo type.}
\newblock{\em Lecture Notes in Mathematics} {\bf 2004}.
\newblock{Springer-Verlag, Berlin, 2010.}
%

\bibitem{Gorenflo_2014}
R.~Gorenflo, A.A.~Kilbas, F.~Mainardi and S.V.~Rogosin.
\newblock{\em Mittag-Leffler Functions, Related Topics and Applications.}
\newblock{Springer Monographs in Mathematics.}
\newblock{Springer-Verlag, Berlin, 2014.}
%

\bibitem{Kilbas}
A.A.~Kilbas, H.M.~Srivastava and J.J.~Trujillo.
\newblock{\em Theory and Applications of Fractional Differential Equations.}
\newblock{North-Holland Mathematics Studies \textbf{204}.}
\newblock{Elsevier Science B.V., Amsterdam, 2006.}
\bibitem{Lancaster}
P.~Lancaster and M.~Tismenetsky.
\newblock{\em The theory of matrices. Second Edition.}
\newblock{Academic Press, San Diego, 1985.}
%

%
\bibitem{Li_2011}
C.P.~Li, and F.R.~Zhang.
\newblock{A survey on the stability of fractional differential equations.}
\newblock{\em Eur. Phys. J. Special Topics,} {\bf 193}(2011), 27--47.
%
\bibitem{Matignon}
D.~Matignon.
\newblock{Stability results for fractional differential equations with applications to control processing.}
\newblock{\em Computational Eng. in Sys. Appl.,} {\bf 2}(1996), 963--968.

\bibitem{Podlubny}
I.~Podlubny.
\newblock{\em Fractional Differential Equations. An Introduction to Fractional Derivatives, Fractional Differential Equations, to Methods of their Solution and some of their Applications.}
\newblock{Mathematics in Science and Engineering \textbf{198}.}
\newblock{Academic Press, Inc., San Diego, CA, 1999}.
%
\bibitem{Qian}
Deliang~Qian, Changpin~Li, Ravi P.~Agarwal and Patricia J.Y.~Wong.
\newblock{Stability analysis of fractional differential systems with Riemann-Liouville derivative.}
\newblock{\em Mathematical and Computer Modeling,} {\bf 52}(2010), 862-874.
%
\bibitem{Sabatier}
J.~Sabatier, M.~Moze and C.~Farges.
\newblock{LMI stability conditions for fractional order systems.}
\newblock{\em Computers and Mathematics with Applications,} {\bf 59}(2010), 1594--1609.
%

\bibitem{Tisdell2012}
C.C.~Tisdell.
\newblock{On the application of sequential and fixed-point method to fractional differential equations of arbitrary order.}
\newblock{\em Journal of Integral Equations and Applications,} {\bf 24}(2012), No 2, 283--319.
%
\bibitem{Wen_2008}
X.~Wen, Z.~Wu, and J.~Lu.
\newblock{Stability analysis of a class of nonlinear fractional--order systems.}
\newblock{\em IEEE Transaction on Circuits and Systems--II: Express Briefs,} {\bf 55}(2008), No 11, 1178--1182.
%
\end{thebibliography}
\end{document}